\documentclass[11pt]{amsart}

\usepackage[left=1in, right=1in, top=1.2in, bottom=1.2in]{geometry}
\usepackage{microtype, mathtools, mathrsfs, enumerate, dsfont, esvect, tikz}
\usepackage{verbatim}
\usepackage[linktocpage]{hyperref}
\hypersetup{
    colorlinks=true,        % false: boxed links; true: colored links
    linkcolor=red,          % color of internal links (change box color with linkbordercolor)
    citecolor=blue,         % color of links to bibliography
    filecolor=magenta,      % color of file links
    urlcolor=cyan           % color of external links
}

\usepackage{amssymb}
\usepackage{url}
\newtheorem{theorem}{Theorem}[section]
\newtheorem{proposition}[theorem]{Proposition}

\newtheorem{lemma}[theorem]{Lemma}

\newtheorem{corollary}[theorem]{Corollary}

\newtheorem{convention}[theorem]{Convention}
\theoremstyle{definition}

\theoremstyle{plain}
\numberwithin{equation}{theorem}

\theoremstyle{remark}

\newtheorem{remark}[theorem]{Remark}

\newif\ifhascomments \hascommentstrue
\ifhascomments
  \newcommand{\dragos}[1]{{\color{red}[[\ensuremath{\bigstar\bigstar\bigstar} #1]]}}
  \newcommand{\matt}[1]{{\color{red}[[\ensuremath{\spadesuit\spadesuit\spadesuit} #1]]}}
\else
  \newcommand{\dragos}[1]{}
  \newcommand{\matt}[1]{}
\fi

\begin{document}

\title{On the growth of algebras, semigroups, and hereditary languages}

\author{Jason Bell}
\address{Jason Bell\\
University of Waterloo\\
Department of Pure Mathematics\\
200 University Avenue West\\
Waterloo, Ontario \  N2L 3G1\\
Canada}
\email{jpbell@uwaterloo.ca}

\author{Efim Zelmanov}
\address{University of California, San Diego\\
Department of Mathematics\\
9500 Gilman Dr.\\
La Jolla, CA 92093}
\email{ezelmanov@ucsd.edu}
\begin{abstract} 
We determine the possible functions that can occur, up to asymptotic equivalence, as growth functions of semigroups, hereditary languages, and algebras.  \end{abstract}

%16D60 %Global ground fields
%16A20, %Linear algebraic groups over arbitrary fields
%16A32. %Iteration problems


\thanks{The first-named author was partially supported by NSERC grant RGPIN-2016-03632.}

\maketitle

%\setcounter{tocdepth}{1}
%\tableofcontents
\section{Introduction}
Let $S$ be a semigroup generated by a finite subset $X$. Consider the function $d_X(n)$, defined to be the number of distinct elements in $\bigcup_{k=1}^n X^k$.  Then $d_X(n)$ is a weakly increasing function and this is called the \emph{growth function} of $S$ with respect to the generating set $X$. If $Y\subseteq S$ is another finite generating subset then since $X^c\supseteq Y$ and $Y^c\supseteq X$ for some $c\ge 1$, we get the inequalities $$d_X(n)\le d_Y(cn)\qquad {\rm and}\qquad d_Y(n)\le d_X(cn).$$ In light of this fact, it is more natural to only consider functions up to asymptotic equivalence. 

Given two weakly increasing functions $f, g:\mathbb{N}\to [1,\infty)$, we say that $f$ is \emph{asymptotically greater than or equal to} $g$ (written $f\succeq g$ or $g\preceq f$), if there is a positive integer $C$ such that $g(n)\le f(Cn)$ for all $n$.  If $g\preceq f$ and $f\preceq g$ then we say that the functions $f$ and $g$ are \emph{asymptotically equivalent} ($f\sim g$).  If $X$ and $Y$ are finite generating subsets of $S$ then $d_X(n)\sim d_Y(n)$.  Thus, regardless of choice of generating set for $S$, the growth function lies in a fixed equivalence class and hence we may speak unambiguously of the growth function of $S$.

Let $X$ be a finite alphabet, let $X^*$ be the set of all words in $X$, and let $W\subseteq X^*$. The hereditary language $L_W(X)$ is defined as the set of all words in $X$ that do not contain subwords lying in $W$.  The function $d(n)$ that counts all words in $L_W(X)$ of length $\le n$ is called the \emph{growth function} of $L_W(X)$ (see \cite{GS81,Gr95}).  The language $L_W(X)$ can be identified with the set of all nonzero elements in the monomial semigroup with zero $\langle X | W\rangle$.  

For algebras, one can produce analogous functions as follows. If $A$ is a finitely generated algebra over a field $k$ and $V$ is a finite-dimensional $k$-vector space that generates $A$ as a $k$-algebra, then one can produce a function $d_V(n)={\rm dim}_k(V^n)$, where $V^n$ is the subspace of $A$ formed by taking the span of all $m$-fold products of in $V$, with $1\le m\le n$. As above, for any two finite-dimensional generating subspaces $V$ and $W$ of the algebra $A$, we have $d_V(n)\sim d_W(n)$.  

A basic question in the theories of the above classes is: \emph{which functions can be realized as growth functions?} We start with the following two necessary conditions for a growth function $f(n)$:
\begin{enumerate}
\item it is weakly increasing;
\item it is \emph{submultiplicative}, i.e., $f(m+n)\le f(m)f(n)$ for all $m,n$.
\end{enumerate}
In the case of groups, Gromov's work \cite{Gro}, combined with works of Bass \cite{Bass} and Guivarc'h \cite{Guivarch}, shows that a polynomially bounded growth function is asymptotically equivalent to $n^d$ for some nonnegative integer $d$. Grigorchuk \cite{Gri} gave the first example of a group whose growth is superpolynomial but subexponential. An important conjecture of Grigorchuk is that there are no groups whose growth is superpolynomial but bounded above by $\exp(n^{\alpha})$ for some $\alpha<1/2$.  Grigorchuk \cite{Gri89} proved this conjecture in the case of residually-$p$ groups.  

Let $\alpha$ be the positive root of the equation $$x^3-x^2-2x-4=0,\qquad \alpha\approx 2.46,$$ and let 
$\beta=\log_{\alpha}(2)\approx 0.767$.  Bartholdi and Erschler \cite{BE} showed that if a function $f(n)$ with 
$$\exp(n^{\beta})\preceq f(n)\preceq \exp(n)$$ satisfies conditions (i) and (ii) above, along with some additional conditions, then $f(n)$ is asymptotically equivalent to the growth function of a group.  

Let's now shift to algebras over fields. Bergman \cite{Berg} (see also \cite[Theorem 2.5]{KL}) proved that if the growth function of an algebra grows super-linearly then its growth must be at least quadratic. Examples of Borho and Kraft \cite{BK} show that for any $\alpha\in [2,\infty)$ one can find a growth function of an algebra (in fact, even of a hereditary language) that is asymptotically equivalent to $n^{\alpha}$. 

Smoktunowicz and Bartholdi \cite{SB} proved that an arbitrary submultiplicative increasing function $f(n)\succeq n^{\log(n)}$ is asymptotically equivalent to a growth function of an algebra.
Greenfeld \cite{Gre17} showed that if a function lies in the segment 
$$\exp(n^{\alpha})\preceq f(n)\preceq \exp(n)$$ for some $\alpha>0$ and satisfies (i), (ii), along with some additional conditions, then $f(n)$ is asymptotically equivalent to the growth function of a finitely generated simple algebra.

It is clear that the class of growth functions of hereditary languages lies in the class of growth functions of semigroups, which, in turn, lies in the class of growth functions of algebras.  In fact, all these three classes coincide.

Given a map $F:\mathbb{N}\to \mathbb{N}$, let $F'(n)=F(n)-F(n-1)$, $F'(0)=F(0)$, be the (discrete) derivative of $F(n)$.  Our main result completely characterizes the functions that can occur as the growth functions of algebras, semigroups, and hereditary languages.

\begin{theorem}
A growth function of an algebra is asymptotically equivalent to a constant function, a linear function, or a weakly increasing function $F:\mathbb{N}\to \mathbb{N}$ with the following properties:
\begin{enumerate}
%\item if $g(n)=G(n)-G(n-1)$ then $g(m+n)\le g(m)g(n)$ for $m,n\ge 0$;
\item $F'(n)\ge n+1$ for all $n$;
\item $F'(m)\le F'(n)^2$ for $n\ge 1$ and $m\in\{n,\ldots ,2n\}$.  
\end{enumerate}
\label{thm: main}
Conversely, if $F(n)$ is either a constant function, a linear function, or a weakly increasing function with the above properties then it is asymptotically equivalent to the growth function of a hereditary language.  In particular, this gives a complete characterization of the functions that can occur as the growth function of an algebra, a semigroup, and of a hereditary language.
\end{theorem}
One can interpret this theorem as saying that other than the necessary condition $F'(n+i)\le F'(n)^2$ for $n\ge 1$ and $i\in\{n,\ldots ,2n\}$, which is related to submultiplicativity, the only additional constraints required for being realizable as a growth function of an algebra are those coming from the gap theorems of Bergman and the elementary ``gap'' that one cannot have strictly sublinear growth that is not constant.

The outline of this paper is as follows.  In \S2, we show that the conditions given in the statement of Theorem \ref{thm: main} are indeed necessary to be the growth function of an algebra with super-linear growth. In \S3, we introduce a combinatorial sequence, which will play a fundamental role in our construction, and we study its basic asymptotic properties.  In \S4, we use the results of \S3 to show that any function having the above properties in Theorem \ref{thm: main} is indeed asymptotically equivalent to the growth function of a hereditary language.

\section{An additional property of growth functions}
In this section, we look at growth functions of finitely generated associative algebras and show that they must satisfy certain inequalities. Given a function $f:\mathbb{N}\to \mathbb{N}$, we can construct a \emph{global counting function}, $F:\mathbb{N}\to \mathbb{N}$, defined by
$F(n)=f(0)+\cdots +f(n)$.  In general, throughout the paper, when working with maps from $\mathbb{N}$ to $\mathbb{N}$ we will generally use lowercase roman letters for the map and the corresponding uppercase roman letter for the global counting function.  Our main result of this section is the following equivalence.

\begin{proposition}
Let $k$ be a field and let $A$ be a finitely generated $k$-algebra and let $V$ be a finite-dimensional $k$-vector space that contains $1$ and that generates $A$ as a $k$-algebra.  If $F(n)={\rm dim}(V^n)$ and $f(n)=F(n)-F(n-1)$ for $n\ge 1$ then 
$f(m)\le f(n)^2$ for $n\ge 1$, $m\in \{n,\ldots ,2n\}$.
\label{prop: reduction}
\end{proposition}
\begin{proof}
We may assume that $A=k\{x_1,\ldots ,x_d\}/I$ and that $V$ is the image of the space $$k+kx_1+\cdots +kx_d$$ in $A$. We impose a degree lexicographic order on words over the alphabet $\{x_1,\ldots ,x_d\}$ by declaring that $x_1<\cdots <x_d$. Then every nonzero element $f\in k\{x_1,\ldots ,x_d\}$ has an initial monomial, which we denote ${\rm in}(f)$, which is the maximum of all words that occur in $f$ with nonzero coefficient.  In particular, there is some nonzero $c\in k$ such that $f=c\cdot {\rm in}(f)+f_0$, where $f_0$ is a linear combination of words that are degree lexicographically less than ${\rm in}(f)$. We let ${\rm in}(I)$ denote the ideal generated by elements ${\rm in}(f)$ as $f$ ranges over the nonzero elements of $I$, and we let $B$ denote the monomial algebra $k\{x_1,\ldots ,x_d\}/{\rm in}(I)$.  Then it is straightforward consequence of the fact that we are using a degree lexicographic order and the theory of Gr\"obner-Shirshov bases (see, for example, \cite{BC}) that
if $W$ is the image of the space $k+kx_1+\cdots +kx_d$ in $B$ then $F(n)$, the dimension of $V^n$, is equal to the dimension of $W^n$, and, moreover, this is precisely the number of words over the alphabet $\{x_1,\ldots ,x_d\}$ of length at most $n$ that do not have any word in ${\rm in}(I)$ as a subword.  Hence $f(n)$ is precisely the number of words over the alphabet $\{x_1,\ldots ,x_d\}$ of length exactly $n$ that do not have any word in ${\rm in}(I)$ as a subword. In particular, growth functions can be completely understood in terms of monoid algebras of finitely generated monoids.  

Notice that if we have a monoid on generators $y_1,\ldots ,y_d$ and we let $g(n)$ denote the number of distinct nonzero words of length $n$ in $y_1,\ldots ,y_d$ then we must have
$g(n+i)\le g(n)^2$ for all $i=0,\ldots ,n$. To see this, observe that to a nonzero word $w$ of length $n+i$ over the alphabet $\{ y_1,\ldots ,y_d\}$, it has a prefix $w'$ of length $n$ and a suffix $w''$ of length $n$, and notice that $w$ is completely determined by $n+i$, $w'$, and $w''$, since $n+i\le 2n$.  Thus we obtain the result $g(n+i)\le g(n)^2$.  The result follows.
\end{proof}

\section{The function $T(d,n)$ and preliminary estimates}
In this section, we introduce a combinatorial function $T(d,n)$, which enumerates certain collections of monomials, and we prove basic asymptotic results concerning this function, which we collect in a series of lemmas.  

Given natural numbers $d$ and $n$ with $1\le d\le n$, we let $T(d,n)$ denote the collection of monomials of length $n$ of the form
$x^i y x^{a_1} y x^{a_2} y \cdots x^{a_s} y x^j$ in which $a_1,\ldots ,a_s\ge d$ with $s\ge 0$ along with the monomial $x^n$.  We find it convenient to introduce the following notation: given a power series ${\sf A}(t)\in R[[t]]$, with $R$ a ring, we let $[t^m]{\sf A}(t)$ denote the coefficient of $t^m$ in ${\sf A}(t)$. 
Then the cardinality of the set of elements in $T(d,n)$ with $s$ occurrences of $y$, $s\ge 1$,
is 
$$[t^{n-s}](1+t+t^2+\cdots )(t^d+t^{d+1}+\cdots )^{s-1} (1+t + t^2+\cdots)$$ and so
$$\#T(d,n)=1+\sum_{s=1}^{\infty} [t^{n-s}] (1-t)^{-2} (t^d/(1-t))^{s-1},$$
which is 
$$1+[t^n] \sum_{s=1}^{\infty} t^s t^{d(s-1)}/(1-t)^{s+1}.$$  Observe that
$$\sum_{s=1}^{\infty} t^s t^{d(s-1)}/(1-t)^{s+1} = \frac{t}{(1-t)^2} \cdot \frac{1}{(1-t^{d+1}/(1-t))} = t(1-t)^{-1}(1-t-t^{d+1})^{-1}.$$
Let
\begin{equation}
h_{d+1}(m):=[t^m](1-t-t^{d+1})^{-1}.
\end{equation}
Then we see that \begin{equation}\label{closed form}
\#T(d,n) = 1+h_{d+1}(0)+\cdots + h_{d+1}(n-1).\end{equation}
\vskip 2mm
\begin{convention} We take $T(i,n)$ to be the set $T(0,n)$ when $i<0$.
\end{convention}
Now we have the following straightforward estimates.  
\begin{lemma} For nonnegative integers $d$ and $n$ with $d<n$ and $n\ge 2$, we have
$$n\cdot \#T(d,2n) \ge \#T(d,n)^2.$$
\label{unequal}
\end{lemma}
\begin{proof}
We have shown that
$\#T(d,n) = 1+h_{d+1}(0)+\cdots + h_{d+1}(n-1)$, where $h_s(m)$ is the coefficient of $t^m$ in $(1-t-t^s)^{-1}$.  Combinatorially, $h_{d+1}(m)$ is just enumerating the number of words of degree $m$ in the free monoid generated by $\{u,v\}$, where $u$ has degree $1$ and $v$ has degree $d+1$.  In particular, considering concatenation of two words of degree $m$ shows that 
\begin{equation} h_{d+1}(2m)\ge h_{d+1}(m)^2,\end{equation} and by considering appending $u$ to the set of words of degree $n$ gives 
\begin{equation} h_{d+1}(n)\le h_{d+1}(n+1).\end{equation}
Thus 
\begin{align*} \#T(d,2n) & = 1+h_{d+1}(0)+\cdots + h_{d+1}(2n-1) \\
&\ge  1+2(h_{d+1}(0)+h_{d+1}(2)+\cdots + h_{d+1}(2n-2)) \\
&\ge  2(1^2 + h_{d+1}(0)^2+h_{d+1}(1)^2+\cdots + h_{d+1}(n-1)^2)-1.\end{align*}
By the Cauchy-Schwarz inequality we have
\begin{align*}
\#T(d,n)^2 =& (1+h_{d+1}(0)+\cdots + h_{d+1}(n-1))^2 \\
&\le  (1^2 + h_{d+1}(0)^2+h_{d+1}(1)^2+\cdots + h_{d+1}(n-1)^2)(n+1)\\
& \le  (\#T(d,2n)+1)(n+1)/2 \\
&\le  \#T(d,2n)\cdot n,\end{align*}
where the last step follows from the fact that $n\ge 2$ and $\#T(d,n)\ge 3$.  The result follows.
\end{proof}

\begin{lemma}
\label{lem2}
Let $d$ and $n$ be natural numbers with $d<n$ and $n\ge 64$.  If $T(d,n)\le n^{4/3}$, then
$$\#T(d,n)^2 \le 2n^{2/3}\cdot \#T(d,2n).$$
\end{lemma}
\begin{proof}
Notice that if $d\le n/2$, then $T(d,n)$ is at least the number of monomials of the form $x^i y x^j y x^k$ with $i+j+k=n-2$ and with $j\ge d$, which is ${n-d\choose 2} > 
n^2/16\ge n^{4/3}$ since $n\ge 64$ and $d\le n/2$.   Thus we may assume without loss of generality that $n/2 < d < n$.  In this case, it is straightforward to compute that $
\#T(d,n)={n-d\choose 2}+n+1\ge (n-d)^2/2$. Thus if 
$\# T(d,n)\le n^{4/3}$ then we must have $(n-d)\le \sqrt{2}\cdot n^{2/3}$ and so $n-\sqrt{2} \cdot n^{2/3}\le d < n$.   Now $T(d,2n)$ contains the set of monomials of the form $x^i y x^j y x^k$ with $i+j+k=2n-2$ and $j\ge d$, which has size ${2n-d\choose 2}$, which is greater than or equal to $n^2/2$, since $d<n$.  Thus 
$$\#T(d,2n) 2n^{2/3} \ge n^{8/3}\ge \# T(d,n)^2$$ in this case, and the result follows.   
\end{proof}

\begin{lemma}
\label{lem3}
Let $n$ be a nonnegative integer. If $d\ge 1$ then $\#T(d,N)\ge \#T(d-1,n)$ whenever $N\ge 2+(n-1)(d+1)/d$.  In addition, $\#T(0,N)\ge \#T(0,n)$ whenever $N\ge n$.\end{lemma}
\begin{proof}
The fact that $\#T(0,N)\ge \#T(0,n)$ whenever $N\ge n$ is immediate, so we may consider the first statement and let $d\ge 1$.
We have shown that
$\#T(d,n) = 1+h_{d+1}(0)+\cdots + h_{d+1}(n-1)$, where $h_s(m)$ is the coefficient of $t^m$ in $(1-t-t^s)^{-1}$.  Combinatorially, $h_{d+1}(m)$ is just enumerating the number of words of degree $m$ in the free monoid generated by $\{u,v\}$, where $u$ has degree $1$ and $v$ has degree $d+1$.  Similarly, $h_{d}(m)$ is the number of words of degree $m$ in the monoid generated by $\{u,v'\}$ where $v'$ has degree $d$.  Notice that there is an injection from the set of words in the monoid generated by $u,v'$ of degree $n$ into the set of words in the monoid generated by $u,v$ of degree $\lceil n(d+1)/d\rceil$ given by replacing all copies of $v'$ by $v$ and then appending the necessary number of $u$'s at the end to make the degree equal to $\lceil n(d+1)/d\rceil$.  Thus
$h_d(m)\le h_{d+1}(\lceil m(d+1)/d\rceil)$.  Hence if $N-1\ge \lceil (n-1)(d+1)/d\rceil$ then we have
\begin{align*}
\#T(d-1,n) &= 1 + h_d(0)+\cdots +h_d(n-1) \\
&\le  1 + \sum_{j=0}^{n-1} h_{d+1}(\lceil j(d+1)/d\rceil) \\
&\le  1 +  \sum_{j=0}^{N-1} h_{d+1}(j)\\
&= \#T(d,N).\end{align*}  In particular, this holds whenever $N\ge 2+(n-1)(d+1)/d$.

\end{proof}

\begin{lemma} Let $d$ and $n$ be natural numbers with $d\le n-1$.  Then $\#T(d,4n)\ge (4n)^{4/3}$ for all $n\ge 512$.
\label{base}
\end{lemma}
\begin{proof}
If $\#T(d,n)\ge n^{4/3}$, then if we use Lemma \ref{unequal}, we see that
\begin{align*}
\#T(d,4n) &\ge  \#T(d,2n)
 \ge  \#T(d,n)^2/n \ge  n^{5/3}
 \ge 
(4n)^{4/3}\end{align*}
 for $n\ge 512$.  
If $\#T(d,n)< n^{4/3}$ then by Lemma \ref{lem2}, using the fact that $\#T(d,n)\ge n+1$, we have
$$\#T(d,2n)\ge \#T(d,n)^2 n^{-2/3}/2 \ge (n+1)^2 n^{-2/3}/2 \ge n^{4/3}/2.$$
Now there are two cases.
First, if  $\#T(2n,n)\le (2n)^{4/3}$, then 
$n^{4/3}/2\le \#T(2n,n)\le (2n)^{4/3}$, and so using Lemma \ref{lem2}, we see
$$\#T(d,4n)\ge   \#T(d,2n)^2 (2n)^{-2/3}/2 \ge n^{8/3} (2n)^{-2/3}/8 \ge (4n)^{4/3}$$ for $n\ge 512$.
If $\#T(d,2n)>(2n)^{4/3}$ then by Lemma \ref{unequal} we have
$$\#T(d,4n)\ge \#T(d,2n)^2/(2n) \ge (2n)^{8/3}\cdot (2n)^{-1} \ge (4n)^{4/3}$$ for $n\ge 256$. The result follows.
\end{proof} 
\begin{lemma} Let $n$ and $d$ be natural numbers with $d\le n$.  For $n\ge 2^{19}$ we have 
$\#T(d,64 n)\ge 512 n \#T(d-1,n)^2$.
\label{lem:512}
\end{lemma}
\begin{proof} We prove this by induction on $n$.
By Lemma \ref{base},
%\ref{lem3}, 
we have $\#T(d-1,4n)\ge (4n)^{4/3}$ and so repeatedly applying Lemma \ref{unequal} gives
\begin{align*}
\#T(d-1,64n) &\ge   \#T(d-1,32n)^2/(32n) \\
&\ge  \#T(d-1,16n)^4/(2^{13} n^3) \\
& \ge  \#T(d-1,8n)^8/(2^{25}n^7) \\
&\ge  \#T(d-1,4n)^{16}/(2^{41}n^{15}).
\end{align*}
But now using the fact that $\#T(d-1,4n)\ge (4n)^{4/3}$ gives \begin{align*}
\#T(d-1,4n)^{16} & \ge  \#T(d-1,4n)^2\cdot \#T(d-1,4n)^{14}/(2^{41} n^{15})\\
&\ge  \#T(d-1,4n)^2\cdot \frac{ ((4n)^{4/3} )^{14}}{2^{41}n^{15}} \\
&\ge  \#T(d-1,4n)^2 \cdot  \frac{2^{112/3} n^{56/3}}{2^{41} n^{15}} \\
&\ge \#T(d-1,4n)^2 \cdot n^{11/3} 2^{-11/3}\\ 
& \ge  512 n \#T(d-1,4n)^2  n^{2/3} 2^{-38/3}.\end{align*}
Since $n^{2/3} 2^{-38/3}\ge 1$ for $n\ge 2^{19}$, we get the result.
\end{proof}

\begin{lemma} Let $d,t,$ and $n$ be natural numbers.  Then 
$$\#T(d,64\cdot 2^t n)\ge (512\cdot 2^t n) \#T(d-1,n)^{2^t}$$ whenever $n\ge 2^{19}$.
\label{growth}
\end{lemma}
\begin{proof}
We prove this by induction on $t$.  When $t=0$, we see that the claim holds when $n\ge 2^{19}$ by Lemma \ref{lem:512}.  Now suppose that the desired inequality holds for $t<s$ for $n\ge 2^{19}$ and consider the case when $t=s$ and $n\ge 2^{19}$.  Then using Lemma \ref{base}, Lemma \ref{unequal}, and the induction hypothesis, we have
\begin{align*}
\#T(d,64\cdot 2^s n) &\ge  \#T(d,64 \cdot 2^{s-1} n)^2/(64 \cdot 2^{s-1} n) \\
&\ge  \frac{ \left(512\cdot 2^{s-1} n \cdot \#T(d-1,n)^{2^{s-1}}\right)^2}{64\cdot 2^{s-1} n} \\
&\ge  \frac{2^{18} \cdot 2^{2s-2} n^2 \cdot \#T(d-1,n)^{2^s}}{64\cdot 2^{s-1} n} \\
&  \ge   2^{12} \cdot 2^{s-1} n \cdot \#T(d-1,n)^{2^s}\\
&\ge 512 \cdot 2^s n \cdot \#T(d-1,n)^{2^s}.
\end{align*}
The result now follows by induction.
\end{proof}
\section{A general construction and proof of Theorem \ref{thm: main}}
In this section we give a construction that allows us to prove Theorem \ref{thm: main}. Specifically, we let
$f:\mathbb{N}\to \mathbb{N}$ be a sequence with the following properties:
\begin{enumerate}
\item[(1)] $f(m)\le f(n)^2$ for $m\in \{n,n+1,\ldots ,2n\}$;
\item[(2)] $f(n)\ge n+1$ for all $n$.
\end{enumerate}
Then we show that if $F$ is the global counting function of $f$, that is,
\begin{equation}
F(n)= f(0)+\cdots +f(n),
\end{equation}
then there is a hereditary language whose growth function is asymptotically equivalent to $F$.  To do this, we need one remark.
\begin{remark}
\label{rem} Item (1) implies that $f(i)\le f(p)^{2^j}$ for $2^{j-1}p \le i \le 2^j p$.
\end{remark}
\begin{proof} Notice that (1) gives that $f(i)\le f(2^{j-1}p)^2$.  Now $f(2^{j-1}p)\le f(2^{j-2}p)^2$ by (1), and so by induction, we see 
$f(2^{j-1}p)\le f(p)^{2^{j-1}}$ and so $$f(i)\le  \left(f(p)^{2^{j-1}}\right)^2=f(p)^{2^j},$$ as required.
\end{proof}

  Then we show that there is a graded $k$-algebra $A=\bigoplus A_n$ with $A_0=k$ and generated in degree $1$ such that $${\rm dim}(A_n)+\cdots + {\rm dim}(A_0) \sim F(n).$$
  (We recall that $\sim$ here represents asymptotic equivalence and not the usual ``asymptotic to'' in analysis.)
It will be clear from the construction that $A$ is a semigroup algebra of a monomial semigroup and therefore corresponds to a hereditary language.

We remark that we may assume without loss of generality that $F(0)=1$ and $F(1)=3$.  
We now recursively define a weakly increasing sequence of natural numbers $d_n$, a subset $X$ of the natural numbers, a function $a:\mathbb{N}\to \mathbb{N}$ whose global counting function is bounded above by $F$, and a sequence of positive integers $e_n$.  We pick $N\ge 2^{19}$ so that the relevant inequalities from the statements of the lemmas in the preceding section hold for $n\ge N$.  Let $d_n=n$ for $n\le N$, let $e_n=n$ for $n\le N$, we let $a(n)=n+1$ for $n\le N$, and we declare that $X$ contains no elements of size less than $N$.  Letting $A(n)$ denote the global counting function of $a(n)$, we see that $A(j)\le F(j)$ for $j\le N$, as $f(j)\ge j+1$ for all $j$.  

Now suppose that $d_i, e_i$, $a(i)$, and $X\cap [0,i]$ have been defined for $i<n$ for some $n> N$ in such a way that $A(n-1)\le F(n-1)$. Then we take $d_n$ to be the smallest natural number $d$ in $\{d_{n-1},\ldots ,n\}$ such that 
$$A(n-1)+\#T(d,n)\le F(n).$$  Notice that
$F(n)-A(n-1)\ge F(n)-F(n-1) = f(n)\ge n+1=\#T(n-1,n)$, and thus $d_n$ exists.
If $d_n>d_{n-1}$, we declare that $n\in X$; otherwise, $n\not\in X$.

We now define $e_n$.  If $n\in X$, then since $x^n=x^n T(d_n-1,0)\in T(d_n,n)$, there is some smallest $e\ge 0$ such that
\begin{equation}
A(n-1)+\#\left(T(d_n,n)\cup x^eT(d_n-1,n-e)\right) \le F(n).
\end{equation}
  We let $e_n$ denote this smallest value of $e$.  We observe that 
$e_n\ge 1$ since $$A(n-1)+\#\left(T(d_n,n)\cup T(d_n-1,n)\right) = A(n-1)+\#T(d_n-1,n) > F(n),$$ and since $d_n>d_{n-1}$.  
Alternatively, if $n\not\in X$, then we let $e_n$ be the smallest $$e\in \{e_{n-1}+1, e_{n-1}+2,\ldots ,n\}$$ such that
$A(n-1)+\#\left(T(d_n,n)\cup x^e\#T(d_n-1,n-e)\right) \le F(n)$. Again, since $$A(n-1)+\#\left(T(d_n,n)\cup x^n\#T(d_n-1,0)\right) \le F(n),$$ we have that $e_n$ exists.

We make the important remark that 
\begin{equation} 
\label{eq:Xplus}
A(p-1)+\#\left(T(d_p,p)\cup x^{e_{p}-1}\#T(d_p-1,p-e_p+1)\right) > F(p)~~ {\rm for}~p\not\in X, {\rm~when}~e_p>e_{p-1}+1.
\end{equation}

Finally, we take \begin{equation}
a(n)=\#(T(d_n,n)\cup x^{e_n}T(d_n-1,n-e_n)).\end{equation}

Observe that
\begin{equation}
a(n)\ge n+1\qquad {\rm for~all}~n\ge 0
\end{equation}
and
\begin{equation}
A(p-1)+\#T(d_p-1,p)>F(p)\qquad {\rm for}~p\in X.
\end{equation}
Furthermore, since $F(p)-A(p-1)\ge F(p)-F(p-1)=f(p)$, we see that
\begin{equation}
f(p)<\#T(d_p-1,p)\qquad {\rm for}~p\in X.
\label{eq:pinX}
\end{equation}
For our purposes, we may assume that $d_n\to \infty$, since if $d_n$ is eventually equal to a constant $d$, then $\#T(d,n)$ grows exponentially in $n$, and so we get that $F(n)$ is asymptotically equivalent to $2^n$ in this case.  The function $2^n$ is asymptotically equivalent to the growth of the free monoid on two generators, and so we are done in this case.
Then by definition 
\begin{equation} \label{eq:X}
X=\{n\colon d_{n-1}<d_n\},
\end{equation} and as we have just remarked, we may assume without loss of generality that $X$ is infinite.

Now we are ready to construct our algebra.
We first let $S$ denote the collection of words that do not occur as a subword of a word in the union
$$\bigcup_{n\ge 1} T(d_n,n).$$  Then we let 
$S'$ denote the elements 
$$\bigcup_{n\ge 1} x^{e_n} T(d_n-1, n-e_n).$$  Then we take $I$ to be the two-sided ideal of $k\{x,y\}$ generated by all words not in the union of $S\cup S'$ and let $R=k\{x,y\}/I$.  We observe that since 
$d_n$ is weakly increasing and since subwords of length $p-1$ of $T(d_p,p)$ are in $T(d_{p-1},p-1)$.  Similarly, if $p\not\in X$ then $e_p>e_{p-1}$ and $d_p=d_{p-1}$,  and so subwords of length $p-1$ of $x^{e_p} T(d_p-1, p-e_p)$ are contained $x^{e_{p-1}} T(d_{p-1},p-1-e_{p-1})$; if, on the other hand, $p\in X$, then subwords of $x^{e_p}T(d_p-1,p-e_p)$ are contained in
$T(d_{p-1},p-1)$, since $d_p-1\ge d_{p-1}$ in this case.  Hence we see that $S\cup S'$ is closed under the process of taking subwords, and so the dimension of the homogeneous component of $R$ of degree $n$ is simply the number of words of length $n$ in $S\cup S'$.

In particular, by construction, the dimension of the homogeneous piece of $R$ of degree $n$ is equal to $a(n)$; that is,
\begin{equation}
\label{eq:an}
a(n) = \#\left( T(d_n,n)\cup x^{e_n} T(d_n-1,n-e_n) \right) \end{equation} and by construction
\begin{equation} A(n) \le F(n)\qquad {\rm for~all}~n,
\end{equation} 
and $A(n)$ is the growth function of the hereditary language constructed above.  

We now make the following remark.  We define
\begin{equation}
Y:= X\cup \{n\colon n\not\in X~{\rm and} ~e_n>e_{n-1}+1.
\end{equation}
\begin{lemma} For $n\in Y$ we have the inequality
$$F(n) \le  A(n) +  \#T(d_n-1, n-e_n-d_n+1).$$
\label{prop:AF}
\end{lemma}
\begin{proof}
%Moreover, by definition, for $n\in X$ we have
%$$F(n)\le \sum_{i\le n-1} \#T(d_i,i) + \#T(d_n-1,n),$$ and so we have
%\begin{equation}
%\label{eq: AFT}
%A(n)\le F(n) \le \sum_{i\le n-1} \#T(d_i,i) + \#T(d_n-1,n)\qquad {\rm for~}n\in X.
%\end{equation}
For $n\in X$ we have by definition of $d_n$ and $e_n$ that 
\begin{equation}
\label{eq: FT}
F(n) \le A(n-1) + \#\left( T(d_n,n)\cup x^{e_n-1} T(d_n-1, n-e_n+1)\right),
\end{equation}
and the same inequality holds for $n\not\in X$ when $e_n>e_{n-1}+1$ by Equation (\ref{eq:Xplus})
Notice that a word in $T(d_n-1, n-e_n+1)$ either begins with an $x$ or a $y$ and if we have a $y$ we must have at least $d_n-1$ copies of $x$ immediately afterwards and so 
$$\#\left( T(d_n,n)\cup x^{e_n-1} T(d_n-1, n-e_n+1)\right)$$ is less than or equal to
$$ \#\left( T(d_n,n)\cup x^{e_n} T(d_n-1, n-e_n)\right) +\#x^{e_n-1} y x^{d_n-1} T(d_n-1, n-e_n-d_n+1).$$
Combining this fact with Equation (\ref{eq: FT}), we see that we have the inequality
\begin{equation}
\label{eq: FT2}
F(n) \le A(n-1)+ a(n)  + \#T(d_n-1, n-e_n-d_n+1)= A(n) +  \#T(d_n-1, n-e_n-d_n+1)
\end{equation}
whenever $n\in X$ or when $n\not\in X$ but $e_n>e_{n-1}+1$.
\end{proof}
We also need a key technical lemma for our analysis.
\begin{lemma}
\label{lem:AA}
Let $N$ be a natural number such that $d_N\ge 5$ and suppose that $i\in Y$ for $i\in \{N,\ldots ,4N\}$ and that, moreover, for every $i\in \{N,\ldots ,4N-1\}$ we have either $i\in X$ or $e_{i+1}\ge e_i+d_i$.
Then $F(N)\le A(4N)$. 
\end{lemma}
\begin{proof}
Suppose towards a contradiction that this is not the case.
Let $N< r_1 < r_2 < \cdots < r_{\ell} < 4N$  denote the elements in $\{N+1,\ldots ,4N-1\}\cap X$.  
Then since $d_n$ is a weakly increasing sequence that must jump at each of $r_1,\ldots ,r_{\ell}$, we see that $d_i\ge j+5$ for $i\ge r_{j}$.  Moreover, since $e_{i+1}\ge e_i+d_i$ for $i\in \{r_j,\ldots ,r_{j+1}-1\}$ and $e_{r_j}\ge 1$, we see that 
\begin{equation}
e_{r_{j+1}-2} \ge 1+d_{r_j}(r_{j+1}-r_j-2)
\end{equation}
 for $j=0,1,2,\ldots ,\ell$, whenever $r_{j+1}\ge r_j+2$, where we take $r_0=N$ and $r_{\ell+1}=4N$.
 
By Lemma \ref{prop:AF}, since each $i\in \{N,\ldots ,4N\}$ is in $Y$, we have
$$F(N) \le  F(i) \le A(i) +   \#T(d_i-1, i-e_i-d_i+1)$$ for $i\in \{N,\ldots ,4N\}$.  
Now if there is some $i\in \{N,\ldots ,4N-1\}$ such that $i+1\not\in X$ and $i+1\ge 2+ (i-e_i-d_i+1)(d_i+1)/d_i$ then by Lemma \ref{lem3}, we have
$$ \#T(d_i-1, i-e_i-d_i+1)\le \#T(d_i,i+1)=\#T(d_{i+1},i+1)\le a(i+1)$$ and so we see
$$F(N)\le A(i)+a(i+1)=A(i+1)\le A(4N),$$ a contradiction.

Thus we may assume that we have $i+1<2+ (i-e_i-d_i+1)(d_i+1)/d_i$ whenever $i+1\in \{N+1,\ldots ,4N\}\setminus X$.  
In particular, if $j$ is such that $r_{j+1}>r_j+1$ and we take $i=r_{j+1}-2$, then we have 
$e_{r_{j+1}-2} \ge 1+d_{r_j}(r_{j+1}-r_j-2)$ and $d_{r_{j+1}-2}=d_{r_j}\ge j+5$.  Consequently, for $i=r_j-2$, we have
$$(i-e_i-d_i+1)(d_i+1)/d_i \le  (r_j-1 - (j+5)(r_{j+1}-r_j-1))(j+6)/(j+5).$$  By assumption, we have
$$i+1=r_j-1 <  (r_j-1 - (j+5)(r_{j+1}-r_j-1))(j+6)/(j+5),$$ and so
simplifying this inequality yields
\begin{equation}
r_{j+1}-r_j-1 < \frac{(r_j-1)}{(j+5)(j+6)}
\end{equation}
whenever $r_{j+1}>r_j+1$ for $j=0,\ldots ,\ell$.  We note that it holds trivially when $r_{j+1}=r_j+1$, and so the inequality in fact holds for $j\in \{0,\ldots ,\ell\}$. 
%In addition, if we take $i=2N-1$ if $r_{\ell}<2N-1$ then a similar analysis gives that
%$$2N-r_{\ell}-1 <  \frac{(r_{\ell}-1)}{(\ell+3)(\ell+4)},$$ where again the inequality is vacuous when $r_{\ell}=2N-1$, so we may assume it holds regardless of the value of $r_{\ell}$.

 We now telescope and see
\begin{align*}
4N - N -(\ell+1) = \sum_{j=0}^{\ell} (r_{j+1}-r_j-1) & < \sum_{j=0}^{\ell} r_j/((j+5)(j+6)) \\
&< 4N\sum_{j\ge 0} \frac{1}{(j+5)(j+6)} \\
& = 4N/5.
\end{align*}
Hence we must have that \begin{equation}
\ell > 11N/5-1.\end{equation}  

But this now gives that $d_{r_{\ell}} > 11N/5$ and so if we take $i=r_{\ell}$ and apply Lemma \ref{prop:AF}, we see
$$F(N) \le  F(i) \le A(i) +   \#T(d_{i}-1,i-e_i-d_i+1).$$ 
But $d_i-1>11N/5$ and $i-e_i-d_i+1 \le 4N-11N/5=9N/5$, since $i<4N$.
Thus $T(d_i-1,i-e_i-d_i+1)$ consists of the words on $x$ and $y$ of length $i-e_i-d_i+1$ with at most one $y$ since $d_i-1\ge i-e_i-d_i+1$.  Consequently,
$\#T(d_i-1,i-e_i-d_i+1)\le 4N+1$.  But now by construction $a(n+1)\ge n+1$ for all $n$ and so
$$F(N) \le  F(i) \le A(i) +   \#T(d_{i}-1,i-e_i-d_i+1) \le A(i)+a(4N) \le A(4N),$$
a contradiction.  This completes the proof. 
\end{proof}
We now show that $A(n)$ and $F(n)$ are asymptotically equivalent.  Since $A(n)\le F(n)$, it suffices to show that $A(n)$ asymptotically dominates $F(n)$.\begin{theorem}
For all $n$ sufficiently large we have $F(n)\le A(2^{16}n)$.
\label{thm:asym}
\end{theorem}
\begin{proof} We may assume that $d_n\to \infty$ since otherwise $F(n)$ and $A(n)$ are both asymptotically equivalent to $2^n$.
We now assume that $n$ is sufficiently large that $d_n\ge 5$ and we divide the proof into two cases.

\emph{Case I.} There is some $p\in [n,512n]\cap Y$.
\vskip 2mm
%In this case, we have $A(n)\le A(p) \le F(p) \le F(512n)$, so it suffices to show that $F(n)\le A(1024n)$ in this first case.
If there is some $p\in [n,2048n-1]\cap Y$ such that $p+1\not\in Y$, then 
by Lemma \ref{prop:AF}, Equation (\ref{eq:an}), and the equalities $d_{p+1}=d_p$ and $e_{p+1}= e_p+1$, we see
\begin{align*}
F(n)&\le  F(p)\\
& \le   A(p) + \#T(d_p-1,p-e_p-d_p+1)\\
& =  A(p) + \#T(d_{p+1}-1, p+1-e_{p+1}-d_{p+1}+1) \\ &\le 
A(p) +\#T(d_{p+1}-1, p+1-e_{p+1})\\
& \le A(p)+a(p+1)\le A(p+1)\le A(2048n).
\end{align*}
Thus we may assume without loss of generality that whenever $p\in [n,2048n-1]\cap Y$, we have $p+1\in Y$. 
Observe, we can say even more when $d_{p+1}=d_p$: in this case, we have $F(p)\le A(p)+\#T(d_{p}-1, p-e_p-d_{p}+1)$ and this is bounded above by 
$A(p+1)$ whenever $p-e_p-d_p+1\le p+1-e_{p+1}$.  Hence we may assume, in addition, that if $p+1\in Y\setminus X$ that we have
$e_{p+1}> e_p+d_p$.  

We are assuming that there is some $p\in [n,512n]\cap Y$, and by the above remark, we may also assume that $[512n,2048n]\subseteq Y$, and that 
$e_{i+1}>e_i+d_i$ whenever $d_{i+1}=d_i$.  Thus by Lemma \ref{lem:AA} we have $A(2048n)\ge F(512n)\ge F(n)$ and so we obtain the result in this case.

\vskip 2mm
\emph{Case II.} $[n,512n]\cap Y$ is empty.
\vskip 2mm
Since $Y$ contains $X$ we then have $[n,512n]\cap X$ is necessarily empty. In this case we pick the largest $p<n$ with $p\in X$.  Since $X$ is infinite, for large enough $n$ we will have that $p\ge 2^{19}$ and we will be able to apply the lemmas from the preceding section.

Then if $d:=d_p$ the we have $d_j=d$ for $j=p,\ldots ,512n$.  Then we pick the largest $t$ such that $64\cdot 2^t p\le 2^{16} n$.
Then by maximality of $t$ we have $64\cdot 2^{t+1} p>2^{16} n$ and so $2^t p>2^9 n$.

Now we have by Lemma \ref{growth} that
\begin{align*}
A(2^{16} n)-A(p) &=
-a(p) + \sum_{j=p}^{2^{16}n} a(j) \\
& \ge  -a(p)+ \sum_{j=0}^{t-1} \sum_{i=64\cdot 2^j p}^{64\cdot 2^{j+1} p -1} \#T(d,i)\\
&\ge -a(p) + \sum_{j=0}^{t-1} 64\cdot 2^j p \cdot \#T(d,64\cdot 2^j p) \\
&\ge  -a(p) + \sum_{j=0}^{t-1} 64 \cdot 2^j p \cdot \#T(d-1, p)^{2^j} \\
&\ge  -a(p) + \sum_{j=0}^{t-1} 64 \cdot 2^j p \cdot f(p)^{2^j},
\end{align*}
where the last inequality follows from Equation (\ref{eq:pinX}).

Now we recall that $f(i)\le f(p)^{2^j}$ for $2^{j-1} p\le i\le 2^j p$ for $j\ge 1$ by Remark \ref{rem}.
Then $$\sum_{i=2^{j-1}p}^{2^jp-1} f(i) \le 2^j p f(p)^{2^j},$$ for $j\ge 1$ so
  \begin{align*}
A(2^{16} n)-A(p) &\ge  -a(p) + \sum_{j=0}^{t-1} 64 \cdot 2^j p \cdot f(p)^{2^j} \\
&\ge  -a(p) + 64\sum_{j=1}^{t-1}  \sum_{i=2^{j-1}p}^{2^jp -1} f(i) \\
&= -a(p) + 64( F(2^{t-1} p-1)-F(p-1))\\
&\ge  -a(p) + 64 (F(128n) - F(p)).\\ 
\end{align*}
So 
\begin{equation}
\label{eq: 16n}
A(2^{16} n)  \ge A(p-1)+64(F(128 n)-F(p)).
\end{equation}
Now by construction, $p\in X$ so we have
$F(p)\le \sum_{i< p} \#T(d_i, i) + \#T(d_p-1,p)$ and we have
$F(128n) \ge \sum_{i\le 128n} \# T(d_i,i)$ and so
we see
\begin{equation}\label{eq: 16n2}
F(128n)-F(p) \ge \left(\sum_{i=p}^{128n} \#T(d,i)\right) - \#T(d-1,p).
\end{equation}
We also have $A(p-1)-F(p)\ge -\#T(d-1,p)$ and 
so by Equations (\ref{eq: 16n}) and (\ref{eq: 16n2}) we have
 \begin{align*}
A(2^{16} n) & \ge  A(p-1) + 63(F(128n)-F(p)) + F(128n)-F(p) \\
&\ge (A(p-1)-F(p)) + F(128n) + 63\left(\sum_{i=p}^{128n} \#T(d,i)  - \#T(d-1,p)\right)\\
&\ge  -64\#T(d-1,p)  + F(128n) + 63\sum_{i=p}^{128n} \#T(d,i)\\
&\ge  F(128n) -64 \#T(d-1,p)+63\#T(d,128n).\end{align*}
Then applying Lemma \ref{lem:512}, we have 
$$\#T(d,128n)\ge \#T(d,128p) \ge 1024 p\#T(d-1,2p)^2\ge 64\#T(d-1,p),$$ and so in particular
$63 \#T(d,128n)\ge 64 \#T(d-1,p)$.  Hence we have
$$A(2^{16} n) \ge F(128n) \ge F(n),$$
and so we see that $F(n)$ is asymptotically dominated by $A(n)$. 
This completes the proof.
\end{proof}
An immediate corollary of Theorem \ref{thm:asym} and the construction given before this theorem is the following result.
\begin{corollary}
\label{cor:A}
Let $f:\mathbb{N}\to \mathbb{N}$ be a sequence with the following properties:
\begin{enumerate}
\item[(1)] $f(m)\le f(n)^2$ for $m\in \{n,n+1,\ldots ,2n\}$;
\item[(2)] $f(n)\ge n+1$ for all $n$.
\end{enumerate}
Then the global counting function of $f(n)$ is asymptotically equivalent to the growth function of a hereditary language.
\end{corollary}
We are now able to complete the proof of our main result.
\begin{proof}[Proof of Theorem \ref{thm: main}]
The fact that the growth function of an algebra is either eventually constant or linear or satisfies conditions (i) and (ii) in the statement of the theorem follows from Bergman's gap theorem \cite{Berg} and Proposition \ref{prop: reduction}.  It is well-known that every constant function and the function $G(n)=n$ are asymptotically equivalent to growth functions of hereditary languages (for example, the function $G(n)=n$ is asymptotically equivalent to the growth function of the hereditary language $t^*$, the free monoid on a single-letter alphabet; the constant function $G(n)=C$ is asymptotically equivalent to the growth function of the hereditary language corresponding to a finite semigroup).  Hence it suffices to consider the case when $G(n)$ satisfies conditions (i) and (ii), and so Corollary \ref{cor:A} gives that there is a hereditary language whose growth function is equivalent to $G(n)$ in this case.  The result follows.
\end{proof}

\end{document}